\documentclass[11pt,reqno]{amsart}
\usepackage{amsthm}
\usepackage{amssymb}
\usepackage{latexsym}
\usepackage{multicol}
\usepackage{verbatim,enumerate}
\usepackage{amscd}
\usepackage[autostyle]{csquotes}

\usepackage[usenames]{color}
\usepackage{hyperref}
\usepackage{amsmath, amscd}

\advance\textwidth by 1.2in \advance\oddsidemargin by -.6in \advance\evensidemargin by -.6in
\newtheorem*{cor}{Corollary}%[section]
\newtheorem*{lem}{Lemma}
\newtheorem*{prop}{Proposition}

\theoremstyle{definition}
\newtheorem*{defn}{Definition}
\theoremstyle{definition}
\newtheorem{thm}{Theorem}

\newtheorem*{rem}{Remark}

\newtheorem{exam}{Example}

\newcounter{cnt}
 \makeatletter
\def\mydggeometry{\makeatletter\dg@YGRID=1\dg@XGRID=20\unitlength=0.003pt\makeatother}
\makeatother \theoremstyle{remark}

\numberwithin{equation}{section}

\let\bwdg\bigwedge
\def\bigwedge{{\textstyle\bwdg}}

\newcommand{\wt}{\operatorname{wt}}

\newcommand{\nc}{\newcommand}
\newcommand{\rnc}{\renewcommand}

\nc{\cal}{\mathcal} \nc{\goth}{\mathfrak} \rnc{\bold}{\mathbf}

\newcommand{\supp}{\operatorname{supp}}

\renewcommand{\Bbb}{\mathbb}
\nc\bomega{{\mbox{\boldmath $\omega$}}} \nc\bpsi{{\mbox{\boldmath $\Psi$}}}
 \nc\balpha{{\mbox{\boldmath $\alpha$}}}
 \nc\bpi{{\mbox{\boldmath $\pi$}}}
\nc\bmu{{\mbox{\boldmath $\mu$}}} \nc\bcN{{\mbox{\boldmath $\cal{N}$}}} \nc\bcm{{\mbox{\boldmath $\cal{M}$}}} \nc\blambda{{\mbox{\boldmath
$\lambda$}}}\nc\bnu{{\mbox{\boldmath $\nu$}}}

\newcommand{\lie}[1]{\mathfrak{#1}}

\makeatletter
\def\section{\def\@secnumfont{\mdseries}\@startsection{section}{1}%
  \z@{.7\linespacing\@plus\linespacing}{.5\linespacing}%
  {\normalfont\scshape\centering}}
\def\subsection{\def\@secnumfont{\bfseries}\@startsection{subsection}{2}%
  {\parindent}{.5\linespacing\@plus.7\linespacing}{-.5em}%
  {\normalfont\bfseries}}
\makeatother

 \nc{\Hom}{\operatorname{Hom}}
  \nc{\mode}{\operatorname{mod}}
\nc{\End}{\operatorname{End}} \nc{\wh}[1]{\widehat{#1}} \nc{\Ext}{\operatorname{Ext}} \nc{\ch}{\text{ch}} \nc{\ev}{\operatorname{ev}}
\nc{\Ob}{\operatorname{Ob}} \nc{\soc}{\operatorname{soc}} \nc{\rad}{\operatorname{rad}} \nc{\head}{\operatorname{head}}

\def\ann{\operatorname{Ann}}

\nc\bc{\mathbb C}

 \nc{\Cal}{\cal} \nc{\Xp}[1]{X^+(#1)} \nc{\Xm}[1]{X^-(#1)}
\nc{\on}{\operatorname} \nc{\Z}{{\bold Z}} \nc{\J}{{\cal J}} \nc{\C}{{\bold C}} \nc{\Q}{{\bold Q}}
\renewcommand{\P}{{\cal P}}
\nc{\N}{{\Bbb N}} \nc\boa{\bold a} \nc\bob{\bold b} \nc\boc{\bold c} \nc\bod{\bold d} \nc\boe{\bold e} \nc\bof{\bold f} \nc\bog{\bold g}
\nc\boh{\bold h} \nc\boi{\bold i} \nc\boj{\bold j} \nc\bok{\bold k} \nc\bol{\bold l} \nc\bom{\bold m} \nc\bon{\bold n} \nc\boo{\bold o}
\nc\bop{\bold p} \nc\boq{\bold q} \nc\bor{\bold r} \nc\bos{\bold s} \nc\boT{\bold t} \nc\boF{\bold F} \nc\bou{\bold u} \nc\bov{\bold v}
\nc\bow{\bold w} \nc\boz{\bold z} \nc\boy{\bold y} \nc\ba{\bold A} \nc\bb{\bold B}  \nc\bd{\bold D} \nc\be{\bold E} \nc\bg{\bold
G} \nc\bh{\bold H} \nc\bi{\bold I} \nc\bj{\bold J} \nc\bk{\bold K} \nc\bl{\bold L} \nc\bm{\bold m} \nc\bn{\bold N} \nc\bo{\bold O} \nc\bp{\bold
P} \nc\bq{\bold Q} \nc\br{\bold R} \nc\bs{\bold S} \nc\bt{\bold T} \nc\bu{\bold U} \nc\bv{\bold V} \nc\bw{\bold W} \nc\bz{\bold Z} \nc\bx{\bold
x} \nc\KR{\bold{KR}} \nc\rk{\bold{rk}} \nc\het{\text{ht }}

\nc\toa{\tilde a} \nc\tob{\tilde b} \nc\toc{\tilde c} \nc\tod{\tilde d} \nc\toe{\tilde e} \nc\tof{\tilde f} \nc\tog{\tilde g} \nc\toh{\tilde h}
\nc\toi{\tilde i} \nc\toj{\tilde j} \nc\tok{\tilde k} \nc\tol{\tilde l} \nc\tom{\tilde m} \nc\ton{\tilde n} \nc\too{\tilde o} \nc\toq{\tilde q}
\nc\tor{\tilde r} \nc\tos{\tilde s} \nc\toT{\tilde t} \nc\tou{\tilde u} \nc\tov{\tilde v} \nc\tow{\tilde w} \nc\toz{\tilde z} \nc\woi{w_{\omega_i}}
\nc\chara{\operatorname{Char}}
\nc{\buf}{\underline{\bof}} 
\nc{\avee}{\alpha^{\vee}}
\nc{\bLambda}{\mathbf{\Lambda}}
\nc{\bMu}{\mathbf{m}}
\nc{\tleq}{\trianglelefteq}

\begin{document}

\author[Fourier]{Ghislain Fourier}
\address{Mathematisches Institut, Universit\"at zu K\"oln, Germany}
\address{School of Mathematics and Statistics, University of Glasgow}
\email{gfourier@math.uni-koeln.de}

\thanks{G.F. is partially funded by the DFG Priority Programme 1388 \enquote{Representation Theory}}

%\subjclass[2000]{Primary: 14C05,17B69}
\date{\today}

\title{Weyl modules and Levi subalgebras}
\begin{abstract} 
For a simple complex Lie algebra $\lie g$ of classical type we are studying the restriction of modules of the current algebra to the current algebra of a Levi subalgebra of $\lie g$. More precisely, we are studying the highest weight components of simple modules, global and local Weyl modules. We are identifying necessary and sufficient conditions on a pair of a Levi subalgebra and a dominant integral weight, such that the highest weight component of the restricted module is a global (resp. a local) Weyl module.
\end{abstract}
\maketitle
%%%%%%%%%%%%%%%%%%%%%%%%%%%%%%%%%%%%%%%%%%%%%%%%%%%%%%%%%%%%%%%%%%%%%%%%%%%%%%%%%%%%%%%%%%%%%%%%%%%%%%%%%%%%%%%%%%%%%%%%%%%%%%%%%%%%%
%%%%%%%%%%%%%%%%%%%%%%%%%%%%%%%%%%%%%%%%%%%%%%%%%%Introduction%%%%%%%%%%%%%%%%%%%%%%%%%%%%%%%%%%%%%%%%%%%%%%%%%%%%%%%%%%%%%%%%%%

\section*{Introduction}
\noindent
Let $\lie g$ be a finite-dimensional, simple complex Lie algebra of classical type with fixed triangular decomposition. Let $\lie a$ be a Levi subalgebra compatible with the root space decomposition of $\lie g$, in some literature $\lie a$ is called a regular subalgebra of $\lie g$ (\cite{Dyn52}). The induced map for the Cartan components $\pi: \lie h_{\lie a} \longrightarrow \lie h$ induces also a map on the level of dominant integral weights $\pi: P^+ \longrightarrow P^+_{\lie a}$.\\
\noindent
Let $V^{\lie g}(\lambda)$ denote the simple finite-dimensional $\lie g$-module of highest weight $\lambda$. By restricting the action we obtain a finite-dimensional $\lie a$-module whose highest weight component is isomorphic to the simple $\lie a$-module $V^\lie a(\pi(\lambda))$. 
We have the analogous picture in modular representation theory: starting with the Weyl module of highest weight $\lambda$ of $\lie g$, the Weyl module of highest weight $\pi(\lambda)$ for $\lie a$ is isomorphic to the submodule through the highest weight vector.\\
In this paper we extend the study of such restriction functors to the context of current algebras. The current algebra associated to $\lie g$ is the vector space $\lie g \otimes \bc[t]$ equipped with the Lie bracket induced by $[x \otimes p, y \otimes q] := [x,y] \otimes pq$. We study simple module, local and global Weyl modules of $\lie g \otimes \bc[t]$ with respect to the restricted action of $\lie a \otimes \bc[t]$.\\
Simple finite-dimensional modules are tensor products of evaluations of simple $\lie g$-modules\cite{Rao93} and parameterized by current highest weights. See Section~\ref{section-2}, where we will recall that a current highest weight can be identified with a finitely supported function $\psi: \bc \longrightarrow P^+$. We will show that similarly to the classical case, the $\lie a \otimes \bc[t]$ module generated through the highest weight vector of $V^{\lie a}(\psi)$ is simple and isomorphic to $V^{\lie a}(\pi \circ \psi)$ (Lemma~\ref{res-simple}).\\
One of the most interesting features of the category of finite-dimensional $\lie g$-modules is the existence of indecomposable, non-simple objects. Therefore, for the purpose of understanding this category, it is not sufficient to understand the simple modules only but it is necessary also to understand the indecomposables. A classification of the indecomposables is far from being known but in \cite{CP01} a special and very important class was introduced. The authors defined \textit{local Weyl modules} in analogy to modular representation theory, that is to any simple module $V^{\lie g}(\psi)$ they associated a cyclic, indecomposable, integrable, finite-dimensional module $W^{\lie g}(\psi)$ which is maximal (in a certain sense) with these properties. \\
Although the modules have been getting a lot of attention due to their special role within the category, they are of particular interest because of their strong connections to other important branches in representation theory, such as simple modules of quantum affine algebras (\cite{CP01}), Demazure modules (\cite{CL06}, \cite{FoL07}), fusion products (\cite{Nao12}), to name but a few.\\
In recent years local Weyl modules have been introduced in more general settings, for example by replacing $\bc[t]$ with a commutative, associative unital algebra (\cite{FL04}, \cite{CFK10}), by considering hyper loop algebras (\cite{JM07}), by considering twisted current algebras (\cite{CFS08}, \cite{FK12}) or even equivariant map algebras (\cite{FMS12}).\\
Throughout the last decade the so-called \textit{global Weyl module} $W^{\lie g}(\lambda)$, defined in \cite{CP01} as a projective object in the category of integrable $\lie g \otimes \bc[t]$-modules with $\lie g$-weights bounded above by $\lambda$, has appeared to carry a lot of the structure of the local Weyl modules (see Section~\ref{section-3} for more details). It was shown that they are also right modules for a certain symmetric algebra $\ba_{\lambda}^{\lie g}$, a polynomial ring in finitely many (determined by $\lambda$) variables. These global Weyl modules appear in various contexts, some far from the origin, as in symmetric functions (\cite{CL12}), q-Whittaker functions (\cite{BF12}) and others. This might justify the intensive study throughout the last decade (\cite{FL04}, \cite{CFK10}, \cite{FMS11}, \cite{BCM12} to name but a few).

\noindent
As previously stated, in this work we are studying the restriction of local and global Weyl module to current algebras of Levi subalgebras. So let $\lie a \subset \lie g$ be a simple Levi subalgebra of $\lie g$ and $\omega_k$ a fundamental weight for $\lie g$. We will introduce for the pair $(\lie a, \omega_k)$ the properties \textit{locally admissible} and \textit{globally admissible} (Definition~\ref{adm-defn}), and we will see that any globally admissible pair is locally admissible. The first main result is the following:
\begin{thm}\label{main-thm1}
Let $(\lie a, \omega_k)$ be locally admissible and $\psi$ be of $\lie g$-weight $\omega_k$. Then 
\[W^{\lie a}(\pi \circ \psi) \hookrightarrow W^{\lie g}(\psi).\] 
\end{thm}
In the end we are mainly interested in branching rules for Weyl modules. A first formula for the branching of $W^{\lie g}(\psi)$ was given in \cite{CL06} for the special case $\lie{sl}_{n} \subset \lie{sl}_{n+1}$, the subalgebra generated by the first $n-1$ simple root vectors. We will generalize this to semi-simple Levi subalgebras $\lie a$ generated by root vectors of simple roots (of $\lie g$), a special class of regular subalgebras (\cite{Dyn52}). The second main result is the following:
\begin{thm}\label{main-thm2} Let $\lambda$ be a dominant integral weight for $\lie g$ such that $(\lie a, \lambda)$ is locally admissible.
\begin{enumerate}
\item[(i)] If $\psi$ is of $\lie g$-weight $\lambda$, then $W^{\lie a}(\pi \circ\psi) \hookrightarrow W^{\lie g}(\psi)$.
\item[(ii)] Further, if $(\lie a, \lambda)$ is globally admissible, then $W^{\lie a}(\pi(\lambda)) \hookrightarrow W^{\lie g}(\lambda)$.
\end{enumerate}
In particular, the highest weight component of the restricted module is again a local/global Weyl module.
\end{thm}

Furthermore the considerations of the fundamental weights imply that the statements of the theorem (subject to the corresponding assumptions) might be read as if and only if. That is to say, if the highest weight component of the restricted local (global) Weyl module is again a local (global) Weyl module for $\lie a \otimes \bc[t]$, then the pair is locally (globally) admissible. The assumption that $\lie a$ is generated by root vectors of simple roots seems to be quite restrictive and we conjecture that the statements of Theorem~\ref{main-thm2} are true even if we drop this assumption.\\
Part(i) of the theorem will be proven by reducing this problem to fundamental weights only, in which case it can be deduced from Theorem~\ref{main-thm1}. For the proof of part (ii) we will use $\ba_{\pi(\lambda)}^{\lie a}$ as a natural subalgebra in $\ba_{\lambda}^{\lie g}$. Then we will show that the restricted module is a free $\ba_{\pi(\lambda)}^{\lie a}-$module of the same rank as the global Weyl module $W^{\lie a}(\pi(\lambda))$.\\
\noindent
We conclude with some final remarks: it is straightforward to ask about the restriction of Weyl modules in the case of generalized current algebras (\cite{CFK10}). The main difficulty one faces here is that in this context the $\lie g$-decomposition of local Weyl modules is not known in the required generality. We will see that the proofs of the main theorems heavily depend of this information. One may use the characterization of Weyl module via homological properties, saying that certain extension groups are trivial (for details we refer to \cite{CFK10}) .\\
\noindent
The paper is organized as follows: in Section~\ref{section-1} we will fix basic notations and introduce the properties locally and globally admissible. In Section~\ref{section-2} we will consider the restriction of finite-dimensional simple module of $\lie g \otimes \bc[t]$. In Section~\ref{section-3} we will briefly review necessary results on local and global Weyl modules. The main theorem will be proven in Section~\ref{section-4} and Section~\ref{section-5}.

%%%%%%%%%%%%%%%%%%%%%%%%%%%%%%%%%%%%%%%%%%%%%%%%%%%%%%%%%%%%%%%%%%%%%%%%%%%%%%%%%%%%%%%%%%%%%%%%%%%%
%%%%%%%%%%%%%%%%%%%%%%% subalgebras

\section{Preliminaries}\label{section-1}
\subsection{}\label{defn-sub}
Let $\lie g$ be a simple, finite-dimensional complex Lie algebra of classical type. We fix a triangular decomposition 
\[
\lie g = \lie n^+ \oplus \lie h \oplus \lie n^-.
\]
We denote the set of (positive) roots of $\lie g$ by $R$ (resp. $R^+$). We denote $Q$ the $\bz$-lattice spanned by $R$ and $Q^+$ the $\bz_{\geq 0}$-lattice spanned by $R^+$. The simple roots of $\lie g$ are denoted by $\alpha_i$, $i \in \{1, \ldots,n \} =: I$, the set of simple roots by $\Pi$. For $\alpha \in R$, we denote the corresponding $\lie{sl}_2$-triple by $x_{\alpha}^{\pm}, h_{\alpha}$.\\
The fundamental weights are denoted by $\omega_i, i \in I$ and the set (dominant) integral weights is denoted by $P$ (resp. $P^+$).
We have 
\[ \lie g = \lie h \oplus \bigoplus_{\alpha \in R}  \lie g_{\alpha}.
\]

\noindent
For a subset $R' \subseteq R$, closed under addition and multiplication by $-1$, we denote $\lie a$ the Lie algebra generated by the root spaces corresponding to $R'$:
\[
\lie a :=  \sum_{\alpha \in R'} [\lie g_{\alpha}, \lie g_{-\alpha}] \bigoplus_{ \alpha \in R'} \lie g_{\alpha}.
\]
Then $\lie a$ is the Levi subalgebra of the reductive subalgebra $\lie a + \lie h$ and $\lie a$ is semi-simple by construction. Since there are various definitions of Levi subalgebras in the literature, we define here, that throughout this paper, a Levi subalgebra is a subalgebra of the form as $\lie a$. 
\begin{exam} Let $\lie g \cong \lie{sl}_4$, and $R' = \{ \alpha_1,  \alpha_{2} + \alpha_3, \alpha_{1} + \alpha_2 + \alpha_3,  -\alpha_1,  -\alpha_{2} - \alpha_3,-\alpha_{1} - \alpha_2 - \alpha_3 \}$. Then $\lie a \cong \lie{sl}_3$.
\end{exam} 
We have an induced triangular decomposition
 \[
\lie a = \lie n^+_{\lie a} \oplus \lie h_{\lie a} \oplus \lie n^-_{\lie a}
\]
where $\lie n^{\pm}_{\lie a} \subseteq \lie n^{\pm}$ and $\lie h_{\lie a} \subseteq \lie h$. The simple roots of $\lie a$ are denoted by $\beta_j$, $j \in \{ 1, \ldots, s \} =: J$, and the set of simple roots is denoted by $\Pi_{\lie a}$, the set of roots $R_{\lie a}$ (resp $R^{+}_{\lie a}$). Again, we denote $Q_{\lie a}$ the $\bz$-lattice spanned by $R_{\lie a}$ and $Q^+_{\lie a}$ the $\bz_{\geq 0}$-lattice spanned by $R^+_{\lie a}$. (for more details see Section~\ref{simple-sub}).\\
The fundamental weights of $\lie a$ are denoted by $\tau_j$, $j \in J$ and the set of  (dominant) integral weights is denoted by $P_{\lie a}( P^+_{\lie a})$.\\
\noindent
Since $\lie h_{\lie a} \subseteq \lie h$ and $R^+_{\lie a} \subset R^+$ we have induced maps $\pi: \lie h^* \twoheadrightarrow \lie h_{\lie a}^*$,  $\pi: P \longrightarrow P_{\lie a}$ and $\pi: P^+ \longrightarrow P^+_{\lie a}$.

\subsection{}
We recall some notations and facts from representation theory. Let $V$ be a finite-dimensional $\lie g$-module. Then $V$ decomposes into its weight spaces with respect to the $\lie h$-action
\[
V = \bigoplus_{ \tau \in P } V_\tau,
\]
where $V_\tau = \{ v \in V \; | \; h.v = \tau(h).v \text{ for all } h \in \lie h\}$.

\noindent
The simple finite-dimensional modules are parameterized by dominant integral weights:
\[
V(\lambda)   \leftrightarrow \lambda \in P^+.
\]

\noindent
Let $\lie a \subseteq \lie g$ be a Lie subalgebra as defined in Section~\ref{defn-sub}. Then the simple finite-dimensional $\lie a$-modules are parameterized by $P^+_{\lie a}$, and we denote for $\mu \in P^+_{\lie a}$ the corresponding simple $\lie a$-module by $V^{\lie a}(\mu)$.\\ 
Let $\lambda \in P^+$. Since $V(\lambda)_\lambda$ is one-dimensional and the category of finite-dimensional $\lie a$-modules is semi-simple, we have
\[
U(\lie a).v_\lambda \cong_{\lie a} V^{\lie a}(\pi(\lambda)).
\]

\subsection{}
We introduce the properties locally and globally admissible for a pair $(\lie a, \lambda)$.
\begin{defn}\label{adm-defn}
Let $\lie a$ be simple, then call a pair $(\lie a, \omega_k)$, $1 \leq k \leq n$ \textit{globally admissible} if 
\[
\pi(\omega_k) = \begin{cases}  0 \\  \tau_j, \text{ for some fundamental weight } \tau_j \text{ of }  \lie a \end{cases}.
\]
A pair $(\lie a, \omega_k)$ is called \textit{locally non-admissible} if 
\begin{enumerate}
\item either $\lie a$  is of type $B_s$, $s > 1$, $\epsilon_i + \epsilon_j$ is the unique simple short root of $B_s$ and $i \leq k \leq n-1$
\item or $\lie g$ is of type $C_n$, $\lie a$ is of type $A_s$ and $\pi(\omega_k) \notin \{0 ,\tau_1, \ldots, \tau_s \}$.
\end{enumerate}
In all other cases $(\lie a, \omega_k)$ is called \textit{locally admissible}.

Let $\lie a$ be a semi-simple Levi subalgebra and $\lambda = \omega_{i_1} + \ldots + \omega_{i_j} \in P^+$. Then we call $(\lie a, \lambda)$ locally (globally) admissible if $(\lie a_i, \omega_{i_j})$ is locally (globally) admissible for all $j$ and all simple constituents $\lie a_i$ of $\lie a$.
\end{defn}

We can immediately relate these two properties:

\begin{prop}\label{global-local}
If $(\lie a, \lambda)$ is globally admissible, then $(\lie a, \lambda)$ is locally admissible.
\end{prop}

\subsection{} The aim of the paper is to study certain modules of the current algebra $\lie g \otimes \bc[t]$, the Lie algebra with the bracket induced by
\[
[ x \otimes p, y \otimes q ] = [x,y]_{\lie g} \otimes pq
\]
for $x,y \in \lie g, p,q \in \bc[t]$. 

\noindent
For every $c \in \bc$,  we have an induced surjective map of Lie algebras
\[
\lie g \otimes\bc[t] \twoheadrightarrow \lie g \otimes \bc \cong \lie g\; : \; x \otimes p \mapsto x \otimes p(c).
\]
Given a $\lie g$-module $V$, we denote the $\lie g \otimes \bc[t]$-module obtained by this pullback by $\ev_{c} V$ or $V_c$ for short.

%%%%%%%%%%%%%%%%%%%%%%%%%%%%%%%%%%%%%%%%%%%%%%%%%%%%%%%%%%%%%%%%%%
%%%%%%%%%%%%%%%%%%%%%%%%%%%%%%%%%%%%%%%%%%%%%%%%%%%%%%%%%%%%%

\section{Simple finite-dimensional modules and subalgebras}\label{section-2}
We consider here simple finite-dimensional modules for $\lie g \otimes \bc[t]$. Let $\mathcal{E}$  denote the monoid
\[
\mathcal{E} = \{ \psi: \bc \longrightarrow P^+ \, | \, |\supp(\psi)| < \infty \}.
\]
We set $\wt(\psi) = \sum_{a \in \supp(\psi)} \psi(a) \in P^+$ and for $\lambda \in P^+$:
\[
\mathcal{E}^\lambda = \{ \psi \in \mathcal{E} \, | \, \wt(\psi) = \lambda \}.
\]

\noindent
To each $\psi \in \mathcal{E}$ one can associate a finite-dimensional $\lie g \otimes \bc[t]$-module
\[
V(\psi) = \bigotimes_{a \in \supp(\psi)} \ev_{a} V(\psi(a)).
\]
Let $v_a$ be a generator of $V(\psi(a))_{\psi(a)}$, then $v_\psi = \otimes_{a \in \supp(\psi)} v_a$ is a cyclic generator of $V(\psi)$. It was shown in \cite{Rao93},\cite{Lau10}, \cite{NSS12}, \cite{CFK10}, that $\mathcal{E}$ parameterizes the simple finite-dimensional $\lie g \otimes \bc[t]$-modules up to isomorphism.

\noindent
Similarly we can define $\mathcal{E}_{\lie a}$, $\mathcal{E}_{\lie a}^{\mu}$ for $\mu \in P^+_{\lie a}$. The map $\pi: P^+ \longrightarrow P^+_{\lie a}$ induces maps
\[
\pi: \mathcal{E} \longrightarrow \mathcal{E}_{\lie a} \; , \; \pi: \mathcal{E}^{\lambda} \longrightarrow\mathcal{E}_{\lie a}^{\pi(\lambda)} \; , \; \pi(\psi) := \pi \circ \psi.
\]
For $\phi \in \mathcal{E}_{\lie a}$ let $V^{\lie a}(\phi)$ denote the corresponding simple module.

\begin{lem}\label{res-simple}
Let $\psi \in \mathcal{E}$, then 
\[
U(\lie a \otimes A).v_\psi \cong V^{\lie a}(\pi(\psi)).
\]
\end{lem}
\begin{proof}
Let $\supp(\psi) = \{a_1, \ldots, a_k\} $. Then by definition of $V(\psi)$:
\[
(\lie g \otimes \prod_{i = 1}^k (t - a_i) \bc[t]).V(\psi) = 0.
\]
This implies that $V(\psi)$ is a simple module for the semi-simple Lie algebra
\[
 \lie g \otimes \left(\bc[t]/(\prod_{i = 1}^k (t - a_i)) \right) \cong \bigoplus_{i = 1}^k \lie g \otimes \left( \bc[t]/(t-a_i)\right).
\] 
We have by definition $V(\psi) = \bigotimes_{i = 1}^k \ev_{a_i} V(\psi(a_i))$ and if $v_i$ is a generator of $V(\psi(a_i))_{\psi(a_i)}$, then 
\[
(U(\lie g \otimes  \bc[t]/ \prod_{i = 1}^k (t-a_i))).\otimes v_i = \bigotimes_{i = 1}^k U(\lie g \otimes \bc[t]/(t-a_i)).v_i \cong  \bigotimes_{i = 1}^k  \ev_{a_i} V(\psi(a_i))
\]
This implies that 
\[
\left(U( \lie a \otimes \bc[t])\right). \otimes v_i = \left(U(\lie a \otimes  \bc[t]/\prod_{i=1}^k (t-a_i) \bc[t])\right). \otimes v_i = \bigotimes_{i=1}^k U(\lie a \otimes \bc[t]/(t-a_i)).v_i .
\]
The lemma follows since
\[
\bigotimes U(\lie a \otimes \bc[t]/(t-a_i)).v_i \cong \bigotimes_{i = 1}^k  \ev_{a_i} V^{\lie a}(\pi(\psi(a_i))) \cong V^{\lie a}(\pi(\psi)).
\]
\end{proof}

We remark here: Let $\lambda \in P^+$. Then $V(\lambda)$ is a finite-dimensional $\lie a$-module (obtained through restriction) and hence it decomposes into a direct sum of simple $\lie a$-modules:
\[
V(\lambda) \cong_{\lie a} \bigoplus_{\tau \in P^+_{\lie a}} V^{\lie a}(\tau)^{\oplus c_{\lambda}^{\tau}}.
\]
The proof above implies that $V(\psi)$ decomposes into a direct sum of simple $\lie a \otimes \bc[t]$-modules. The multiplicities $c_{\lambda}^{\tau}$ would immediately give a branching rule for these simple modules. Unfortunately, they are far from being known in general.

\subsection{}
We will relate \enquote{globally admissible} and the set of finitely supported functions with weight $\lambda$:

\begin{prop}\label{surjective-E}
Let $(\lie a, \lambda)$ be globally admissible and $\phi \in \mathcal{E}_{\lie a}^{\pi(\lambda)}$. Then there exists $\psi \in \mathcal{E}^{\lambda}$ such that
\[
\pi(\psi) = \phi \text{ and } V^{\lie a}(\phi) \cong_{\lie a \otimes \bc[t]} U(\lie a \otimes \bc[t]).v_\psi \; . \; 
\]
In particular $\pi: \mathcal{E}^{\lambda} \longrightarrow \mathcal{E}^{\pi(\lambda)}_{\lie a}$ is surjective. 
\end{prop}
\begin{proof}
Let $\lambda = \sum_{i \in I} m_i \omega_i$ and $\pi(\lambda) = \sum_{j \in J} n_j \tau_j$. Suppose we have a decomposition of $\lambda = \mu_1 + \ldots+ \mu_k$, with $\mu_\ell \in P^+$. Then $(\lie a, \lambda)$ is globally admissible if and only if $(\lie a,  \mu_\ell)$ is globally admissible for all $\ell \in \{1, \ldots, k\}$.\\
Since  $(\lie a, \lambda)$ is globally admissible we can find for all $j \in J$, such that $n_j \neq 0$, an $i_j \in I$ with $\pi(\omega_{i_j}) = \tau_j$. Then $i_j \neq i_{j'}$ if $j \neq j'$.\\
This implies that if $\pi(\lambda) = \nu_1 + \ldots + \nu_k$ is a decomposition into dominant integral weights for $\lie a$, then there exist $\mu_1, \ldots, \mu_k$ such that
\begin{eqnarray}
\lambda = \mu_1 + \ldots + \mu_k \; \text{ and } \pi(\mu_\ell) = \nu_\ell \text{ for all } 1 \leq \ell \leq k.
\label{eq:prop-s}
\end{eqnarray}

\noindent
Let $\phi \in \mathcal{E}^{\pi(\lambda)}_{\lie a}$. We have to show that there exists $\psi \in \mathcal{E}^\lambda$ such that $\pi(\psi) = \phi$. We can write $\phi = \sum_{a \in \supp(\phi)} \phi_a$, where
\[
\phi_a (b) := \begin{cases} \phi(a) \text{ if } b = a \\ 0 \text{ else } \end{cases}
\]
It is clearly sufficient to find preimages for all $\phi_a$. By \eqref{eq:prop-s}, for all $a \in \supp(\phi)$ there exists $\mu^a \in P^+$ such that 
\[
\lambda = \sum_{a \in \supp(\phi)} \mu^a \; \text{ and } \pi(\mu^a) = \phi(a).
\]
We define $\psi_a \in \mathcal{E}^{\mu^a}$ by $\psi_a(a) = \mu^a$ and $\psi_a(b) = 0$ for $b \neq a$. Then
\[ \pi(\psi_a) = \phi_a \text{ and } \sum_{a \in \supp(\phi)} \wt(\psi_a) = \lambda.
\]
By setting  \[\psi := \sum_{a \in \supp(\phi)} \psi_a,\] we have $\psi \in \mathcal{E}^{\lambda}$ and $\pi(\psi) = \phi$.
This implies the proposition, since with Lemma~\ref{res-simple}
\[
V^{\lie a}(\phi) \cong_{\lie a \otimes \bc[t]} U(\lie a \otimes \bc[t]).v_\psi.
\]
\end{proof}

\noindent
Let us briefly consider the situation where $(\lie a, \lambda)$ is not globally admissible:
let $\lambda = \sum m_i \omega_i$ and we will assume for the moment that $m_i \neq 0 \Rightarrow \pi(\omega_i) \neq 0$. Further, let $\pi(\lambda) = \sum n_i \tau_i$. Then we have $\sum m_i \leq \sum n_i$.\\
Suppose $(\lie a, \lambda)$ is not globally admissible. Then there exists $i\in I$ with $m_i \neq 0$ and $\pi(\omega_i)$ is not a fundamental weight. This implies in this case $\sum m_i < \sum n_i$.\\
But this also implies that for any $\psi \in \mathcal{E}^\lambda$ we have 
\[|\supp (\psi)| = |\supp (\pi (\psi)) | \leq \sum m_i < \sum n_i.\] 
On the other hand we have the regular functions in $\mathcal{E}_{\lie a}^{\pi(\lambda)}$ (e.g. functions assigning to each point either $0$ or a fundamental weight). For such a regular function $\phi$ we have $|\supp (\phi) | = \sum n_i$. This  implies that there does no exist $\psi \in \mathcal{E}^{\lambda}$ with $\pi(\psi) = \phi$, so  $\pi: \mathcal{E}^{\lambda} \longrightarrow \mathcal{E}_{\lie a}^{\pi(\lambda)}$ is not surjective.\\
It remains to consider the cases where there exists $m_i \neq 0$ and $\pi(\omega_i) = 0$.  Let 
\[ S(\lambda):= \{i \in I \, | \, m_i \neq 0, \pi(\omega_i) = 0 \} \text{ and } \lambda' := \lambda - \sum_{i \in S(\lambda)} m_i \omega_i.\] 
Then we have $\pi(\lambda) = \pi(\lambda')$ and $(\lie a, \lambda)$ is globally admissible if and only if  $(\lie a, \lambda')$ is globally admissible. Moreover $\pi':  \mathcal{E}^{\lambda'} \longrightarrow \mathcal{E}_{\lie a}^{\pi(\lambda')}$ is surjective if and only if  $\pi:  \mathcal{E}^{\lambda} \longrightarrow \mathcal{E}_{\lie a}^{\pi(\lambda)}$ is surjective. Together with the case already considered we conclude for all $\lambda \in P^+$:
\[ (\lie a, \lambda) \text{ is globally admissible } \Leftrightarrow \pi: \mathcal{E}^{\lambda} \longrightarrow \mathcal{E}_{\lie a}^{\pi(\lambda)} \text{ is surjective.}
\] 
%%%%%%%%%%%%%%%%%%%%%%%%%%%%%%%%%%%%%%%%% Local Weyl modules %%%%%%%%%%%%%%%%%%%%%%%%%%%%%%%%%

\section{Global and local Weyl modules}\label{section-3}
In this section we recall the definitions and some facts on global and local Weyl modules.

\subsection{}
Following \cite{CP01} we define
\begin{defn}\label{global-defn}
Let $\lambda \in P^+$, then the $\lie g \otimes \bc[t]$-module $W^{\lie g}(\lambda)$ generated by a non-zero vector $w_\lambda$ subject to the relations
\[ 
\lie n^+ \otimes \bc[t].w_\lambda = 0 \; , \; (h \otimes 1).w_\lambda = \lambda(h)w_\lambda \; , \; (x_\alpha^- \otimes 1)^{\lambda(h_\alpha) + 1}.w_\lambda = 0
\]
for all $h \in \lie h$ and $\alpha \in R^+$, is called the \textit{global Weyl module} of weight $\lambda$.
\end{defn}

\noindent
A couple of remarks are necessary.
\begin{rem}\label{remark-global} We are citing here mainly research of the last decade:
\begin{enumerate}
\item By definition, $W^{\lie g}(\lambda)$ admits a \textit{universal property}. Any cyclic highest weight $\lie g \otimes \bc[t]$-module of highest weight $\lambda$ is a quotient of $W^{\lie g}(\lambda)$.
\item In the definition, $\bc[t]$ can be replaced by a commutative, associative unital algebra over $\bc$. One may see for example \cite{CFK10}, \cite{FL04} to get an insight into the difficulties and differences to the classical situation that one faces here.
\item Global Weyl modules can also be defined for twisted loop algebras (\cite{FMS11}) or more general for equivariant map algebras. (\cite{FMS12})
\item By construction, $W^{\lie g}(\lambda)$ is an integrable $\lie g$-module but infinite dimensional, even the multiplicity of every simple $\lie g$-module appearing in a decomposition is infinite.
\item $W^{\lie g}(\lambda)$ is projective in the category of integrable $\lie g \otimes \bc[t]$-modules with weights bounded above by $\lambda$ (see \cite{CFK10}).
\item $W^{\lie g}(\lambda)$ plays an important role in the context of q-Toda integrable systems, namely its character gives a so-called q-Whittaker function (see \cite{BF12}).
\end{enumerate}
\end{rem}

\noindent
We denote 
\[
\ba_\lambda^{\lie g} := U(\lie h \otimes \bc[t])/\ann_{U(\lie h \otimes \bc[t])}(w_\lambda).
\]
Then $\ba_\lambda^{\lie g}$ is a commutative, associative, unital algebra. Further (\cite{CP01}, resp \cite{CFK10} for the second part):
\begin{thm}\label{ba-lambda}
If $\lambda = \sum_{i \in I} m_i \omega_i$, then
\[
\ba_\lambda^{\lie g} \cong \bigotimes_{i \in I} S^{m_i}(\bc[t])
\]
where $S^{m_i}(\bc[t])$ denotes the $m_i$-th symmetric algebra of $\bc[t]$. Further
\[
\ba_\lambda^{\lie g}  \cong U(\lie h \otimes \bc[t])/I
\]
where 
\[
I = \bigcap_{\psi \in \mathcal{E}^{\lambda}} \ann_{U(\lie h \otimes \bc[t])} (v_\psi).
\]
\end{thm}

\noindent
For more on this algebra, the interested reader may also see \cite{CFK10}, \cite{CL12}.

\noindent
$W^{\lie g}(\lambda)$ admits the structure of a $(\lie g, \ba_{\lambda}^\lie g)$-bimodule, where the right action is given by
\[
u w_\lambda.h \otimes a := u(h \otimes a)w_\lambda
\]
for $u \in  U(\lie g \otimes \bc[t]), h \otimes a \in \lie a \otimes \bc[t]$.

\noindent
With this we have the following theorem which follows from results due to \cite{CP01} for $\lie g \cong \lie{sl}_2$, \cite{CL06} for $\lie g \cong \lie{sl}_{n+1}$, \cite{FoL06} for $\lie g$ of simply-laced type and \cite{Nao12} for all classical $\lie g$.
\begin{thm}\label{global-free}
Let $\lambda \in P^+$, then $W^{\lie g}(\lambda)$ is a free right $\ba_{\lambda}^{\lie g}$-module of finite rank.
\end{thm}

\subsection{}\label{section-local}
The global Weyl module $W^{\lie g}(\lambda)$ induces a functor from the category of $\ba_\lambda^{\lie g}$-modules to the category of integrable $\lie g \otimes \bc[t]$-modules with weights bounded above by $\lambda$:
\[ 
M \mapsto W^{\lie g}(\lambda) \otimes_{\ba_\lambda^{\lie g}} M\; ; \; f \mapsto 1 \otimes f
\]

\noindent
Since $\ba_\lambda^{\lie g}$ is a polynomial ring in finitely many variables, so especially finite generated, any one-dimensional $\ba_\lambda^{\lie g}$-module is isomorphic to $\ba_\lambda^{\lie g}/\bm$ for a uniquely determined maximal ideal of $\ba_\lambda^{\lie g}$. Maximal ideals of $\ba_\lambda^{\lie g}$ are parameterized by $\mathcal{E}^{\lambda}$ (see \cite{CFK10} for details), so to any maximal ideal $\bm$ of $\ba_\lambda^{\lie g}$ there exists a finitely supported function $\psi$ of weight $\lambda$.

\begin{defn} Let $\lambda \in P^+$ and  $\psi \in \mathcal{E}^{\lambda}$, denote $\bm$ be the corresponding maximal ideal of $\ba_\lambda^{\lie g}$. The $\lie g \otimes \bc[t]$-module
\[ 
W^{\lie g}(\lambda) \otimes_{\ba_\lambda^{\lie g}}  \ba_\lambda^{\lie g}/\bm
\]
is called the \textit{ local Weyl module } associated to $\psi$ and denoted by $W^{\lie g}(\psi)$.
\end{defn}

\begin{rem}\label{remark-local} Again, the following remarks might be helpful.
\begin{enumerate}
\item The local Weyl module admits a \textit{universal property} similar to the previous one for global Weyl modules. Let $V$ be a cyclic finite-dimensional  $\lie g \otimes \bc[t]$ highest weight module, generated by a one-dimensional weight space of current weight $\psi$. Then $V$ is a quotient of $W^{\lie g}(\psi)$ (\cite{CP01}).
\item The definition of local Weyl modules is due to \cite{CP01}. They were defined analogously to modular representation theory. In this context the category of finite-dimensional modules is not semi-simple, so one may ask for the maximal object having a certain unique simple quotient.
\item In \cite{CFK10} the construction of local Weyl modules was generalized to an arbitrary commutative algebra $A$ instead of $\bc[t]$. The tensor functor from the category of $\ba_\lambda^{\lie g}$-module to the category of integrable $\lie g \otimes \bc[t]$-modules is called the \textit{Weyl functor}.
\item Local Weyl modules were studied also for the twisted loop algebra in \cite{CFS08}, for the twisted current algebra in \cite{FK12}, for equivariant map algebras (under some restrictions) in \cite{FKKS12}.
\end{enumerate}
\end{rem}

The simple $\lie g \otimes \bc[t]$-modules of highest (classical) weight are also parameterized by $\mathcal{E}^\lambda$ (see Section~\ref{section-2}). The universal property of the local Weyl modules gives the relation between these classes of modules. For any $\psi \in \mathcal{E}^\lambda$ we have a surjective map of $\lie g \otimes \bc[t]$-modules:
\[
W^{\lie g}(\psi) \twoheadrightarrow V(\psi).
\]

\noindent
Theorem~\ref{global-free} implies: 
\begin{lem}\label{pointless}
The dimension and $\lie g$-character of a local Weyl module does not depend on the chosen maximal ideal in $\ba_\lambda^{\lie g}$. In particular, for $\lambda \in \P^+$ and $\psi_1, \psi_2 \in \mathcal{E}^{\lambda}$:
\[
W^{\lie g }(\psi_1) \cong_{\lie g} W^{\lie g}(\psi_2).
\]
\end{lem}

\noindent
Historically the theorem on local Weyl modules was proven first. In fact, Lemma~\ref{pointless} and Theorem~\ref{global-free} are equivalent (for details on this, we refer to the appendix in \cite{FMS12}).

\noindent
The local Weyl modules share the very important \textit{tensor product property}, similar to the one satisfied by the simple modules. It was proven in \cite{CP01}, and further generalized in \cite{FoL07}, \cite{CFK10}:
\begin{thm}\label{tensor}
Let $\lambda_1, \lambda_2 \in P^+$, $\psi_i \in \mathcal{E}^{\lambda_i}$. If $\supp (\psi_1) \cap \supp (\psi_2) = \emptyset$, then
\[
W^{\lie g} (\psi_1 + \psi_2) \cong_{\lie g \otimes \bc[t]} W^{\lie g}(\psi_1) \otimes W^{\lie g}(\psi_2).
\]
\end{thm}

Let us briefly see some consequences. Let $\psi \in \mathcal{E}$ and $\supp (\psi) =  \{c_1, \ldots, c_s \}$, $\psi_c$ being supported on $c$ only such that $\psi = \psi_{c_1} + \ldots + \psi_{c_s}$.  Then is was shown in  \cite{CFS08} that
\[ \lie g \otimes \prod_{i =1}^{s} ( t - c_i)^{N_i} \bc[t] W^{\lie g}(\psi) = 0,
\]
where $N_i \geq \psi_{c_i}(h_{\theta})$, $\theta$ being the highest root of $\lie g$. \\
This implies that $W^{\lie g}(\psi) $ is actually a module for the quotient algebra (the equality below is nothing but an application of the Chinese Remainder Theorem)
\[ \lie g \otimes \bc[t]/ \lie g \otimes \prod_{i =1}^{s} ( t - c_i)^{N_i} \bc[t]  \cong \bigoplus_{i=1}^{s} \lie g \otimes \bc[t]/(t-c_i)^{N_i}\bc[t].\]
Let now $w_{\psi}, w_{\psi_{1}}, \ldots, w_{\psi_{s}}$ be highest weight generators of the Weyl modules 
$W^{\lie g}(\psi)$, $W^{\lie g}(\psi_{c_1})$, $\ldots$, $W^{\lie g}(\psi_{c_s}) $. Then
\[ U(\lie g \otimes \bc[t]).w_{\psi} = U( \lie g \otimes \bc[t]/ \lie g \otimes \prod_{i =1}^{s} ( t - c_i)^{N_i} \bc[t] ).w_\psi.\]
and the right hand side is equals
\[U(\lie g \otimes \bc[t]/(t-c_1)^{N_1} \bc[t]).w_{\psi_1} \bigotimes \ldots \bigotimes U(\lie g \otimes \bc[t]/(t-c_s)^{N_s} \bc[t]) .w_{\psi_s}\]
This is nothing but 
\[  W^{\lie g}(\psi_{c_1}) \otimes \ldots \otimes W^{\lie g}(\psi_{c_s}) . \]
We will use this observation in the following section.\\

To complete this survey on local Weyl modules we should recall the dimension and $\lie g$-character formulas for local Weyl modules. Theorem~\ref{tensor} tells us that the dimension and character of $W^{\lie g}(\psi)$ is nothing but the product of the dimensions and characters of the $W^\lie g(\psi_{c_i})$. \\
Now there are two reductions. First, since these information are independent of the support (Lemma~\ref{pointless}), we may assume that $\psi(c)$ is either $0$ or a fundamental weight. So it suffices to compute these information for local Weyl modules supported in a single point only and the highest weight is a fundamental weight. \\
Second, again since dimension and character are independent of the support (Lemma~\ref{pointless}), we are left with computing the local Weyl module for $\psi$ being supported in $0$ only and $\psi(0) = \omega_i$ (for some $i \in I$). In this case, the defining relations of the local Weyl module become homogeneous and $W^\lie g(\psi)$ is a graded $\lie g \otimes \bc[t]$-module.\\
The following list of the $\lie g$-decomposition of $W^{\lie g}(\psi_i)$ is due to \cite{Cha01} (see also \cite{FoL06}):
\begin{prop}\label{classical-decom}
Let $\lie g$ be a simple Lie algebra of type $X_n$ and $i \in I$, then in the various cases we have the following decomposition of the fundamental graded local Weyl module $W^{\lie g}(\psi_i)$ as a $\lie g$-module
\begin{enumerate}
\item Type $A_n$: 
\[
W^{\lie g}(\psi_i) \cong_{\lie g} V(\omega_i)
\]
\item Type $B_n$:
\[
W^{\lie g}(\psi_i) \cong_{\lie g} V(\omega_i) \oplus V(\omega_{i-2}) \oplus \ldots \oplus V(\omega_{\epsilon}) \text{ for } i < n
\]
\[
W^{\lie g}(\psi_n) \cong_{\lie g} V(\omega_n)
\]
\item Type $C_n$:
\[
W^{\lie g}(\psi_i) \cong_{\lie g} V(\omega_i)
\]
\item Type $D_n$
\[
W^{\lie g}(\psi_i) \cong_{\lie g} V(\omega_i) \oplus V(\omega_{i-2}) \oplus \ldots \oplus V(\omega_{\epsilon}) \text{ for } i < n-1
\]
\[
W^{\lie g}(\psi_i) \cong_{\lie g} V(\omega_i) \text{ for } i = n-1, n
\]
\end{enumerate}
where $\omega_{\epsilon} = \begin{cases} \omega_1 \text{ if } i \text{ is odd } \\ 0 \text{ if } i \text{ is even } \end{cases}$
\end{prop}

\subsection{}\label{fusion-product}
We will recall the definition of fusion products and the construction of local Weyl modules supported in a single point only.\\

\noindent
The following construction is due to \cite{FL99}. We denote for $r \geq 0$
\[
\mathcal{F}^r = \{ u \in U(\lie g \otimes \bc[t]) \; | \; \deg(u) \leq r \}
\]
where the degree of a monomial is the sum of the degree of $t$ in its factors. Then $\mathcal{F}^{0} = U(\lie g)$. We set $\mathcal{F}^{-1} = 0$.\\
Let $W$ be a cyclic $\lie g \otimes \bc[t]$-module generated by a fixed vector $w$. We have an induced filtration on $W$:
\[ 
W^{r} := \mathcal{F}^{r}.w
\]
Then the associated graded module is
\[
W^{gr} = \bigoplus_{r \geq 0} W^{r}/W^{r-1}.
\]
Recall that, since $U(\lie g).\mathcal{F}^{r} = \mathcal{F}^{r}$ for $r \geq 1$, the graded components are $U(\lie g)$-modules.

\noindent
Let $\lambda_1, \ldots, \lambda_k \in P^+$, $\psi_i \in \mathcal{E}^{\lambda_i}$, $\supp (\psi_i) = \{ c_i \}$. Let $w_i$ a generator of the one-dimensional weight space $W^{\lie g}(\psi_i)_{\lambda_i}$. Then 
\[
W^{\lie g}(\psi_1) \otimes \ldots \otimes W^{\lie g}(\psi_k)
\]
is cyclic generated by $w_1 \otimes \ldots \otimes w_k$ if 
\[
c_i \neq c_j \text{ for all } i \neq j.
\]
The associated graded module is called the \textit{ fusion product  } (\cite{FL99}) and denoted by
\[
W^{\lie g}(\psi_1) \ast \ldots \ast W^{\lie g}(\psi_k).
\]
One may also find the notation
\[
W^{\lie g}(\lambda_1)_{c_1} \ast \ldots \ast W^{\lie g}(\lambda_k)_{c_k}
\]
in the literature.
\noindent
The following corollary follows immediately from Lemma~\ref{pointless}.
\begin{cor}\label{local-fusion} Let $\psi \in \mathcal{E}^{\lambda_1 + \ldots + \lambda_k}$ be defined by $\psi(0) = \lambda_1 + \ldots + \lambda_k$, then
\[
W^{\lie g}(\psi_1) \ast\ldots \ast W^{\lie g}(\psi_k) \cong_{\lie g \otimes \bc[t]} W^{\lie g}(\psi)
\]
\end{cor}

\noindent
Let $\psi_c \in \mathcal{E}^{\lambda}$ be supported in a single point only, say $\supp(\psi) = \{c \}$. Then $W^{\lie g}(\psi_c)$ can be constructed from $W^{\lie g}(\psi_0)$ by using the automorphism $\sigma_c$ of $\lie g \otimes \bc[t]$ induced by
\begin{eqnarray}
x \otimes t \mapsto x \otimes (t-c).
\label{eq:pullback}
\end{eqnarray}
Then $W^{\lie g}(\psi_c)$ is nothing but the pullback of $W^{\lie g}(\psi_0)$ via $\sigma_c$.

\noindent
With this, we can construct any local Weyl module $W^{\lie g}(\psi)$ as the tensor product of pullbacks of fusion products of fundamental local Weyl modules.

\section{Local Weyl modules and subalgebras}\label{section-4}
\subsection{}
In this section we prove Theorem~\ref{main-thm1}. Here, $\lie a$ is a simple Levi subalgebra of $\lie g$.
\begin{lem}\label{funda-res} Let $k \in I$ such that $(\lie a, \omega_k)$ is locally admissible and let $\psi \in \mathcal{E}^{\omega_k}$,  $w$ a generator of $W^{\lie g}(\psi)_{\omega_k}$. Then
\begin{eqnarray}
W^{\lie a}(\pi(\psi)) \cong U(\lie a \otimes \bc[t]).w \subseteq W^{\lie g}(\psi)
\label{eq:fund}
\end{eqnarray}

\end{lem}
\begin{proof}
First, let us explain why it suffices to prove the statement for $\psi$ being supported in $0$ only. Clearly, if $\psi \in \mathcal{E}^{\omega_i}$, then $|\supp (\psi) | = 1$, let $\supp (\psi) = \{c\}$. Using the pullback $\sigma_c$ \eqref{eq:pullback} we see that $\sigma_c^* W^\lie g(\psi) \cong W^\lie g(\psi_0)$, where $\psi_0(a) := \psi(a+c)$. But $\sigma_c$ restricts to an isomorphism of $\lie a \otimes \bc[t]$ as well. Acting on both sides of \eqref{eq:fund} with $\sigma_c$ we could deduce the statement of the lemma from the $c = 0$ case.\\
Since $U(\lie a \otimes \bc[t]).w$ is a cyclic $\lie a \otimes \bc[t]$-module of highest weight $\wt(\pi(\psi))$, we can use the universal property (Remark~\ref{remark-local}) of local Weyl modules. So we obtain a surjective map of $\lie a \otimes \bc[t]$-modules: 
\[
W^{\lie a}(\pi(\psi))  \twoheadrightarrow_{\lie a \otimes \bc[t]} U(\lie a \otimes \bc[t]).w.
\]
We also have
\begin{eqnarray}
U(\lie a \otimes \bc[t]).w \twoheadrightarrow_{\lie a \otimes \bc[t]} V^{\lie a}(\pi(\psi))
\label{eq:simple}
\end{eqnarray}
since $V^{\lie a}(\pi(\psi))$ is simple.\\
If $W^{\lie a}(\pi(\psi)) \cong_{\lie a} V^{\lie a}(\tau_s)$, where $\tau_s = \pi(\omega_k)$, then the claim follows. We refer in the following to this as a \textit{simple case}. For instance, if $\lie g$ is of type $A$, then we can read off Appendix~\ref{a2} that $\lie a$ is of type $A$ as well. Moreover we see that $\pi(\omega_k)$ is a fundamental weight for $\lie a$. Buth then Proposition~\ref{classical-decom} tells us that in these cases $W^{\lie a}(\pi(\psi)) \cong_{\lie a} V^{\lie a}(\tau_s)$. Hence if $\lie g$ is of type $A$ then the case is simple for all $\lie a$ and all $\omega_k$.

\noindent
Suppose $\lie g$ is of type $B_n$. We will check the various cases (see the Appendix 2 for notations):
\begin{enumerate}
\item Suppose $\lie a$ is of type $B_s$ and $\epsilon_i + \epsilon_j$ the unique simple short root in $\Pi_{\lie a}$. We will show that 
$ U(\lie a \otimes \bc[t]).w$ has the same classical decomposition as $W^{\lie a}(\pi(\psi)) $ (see Theorem~\ref{classical-decom}). Let $k \in I$:
\begin{enumerate}
\item if $i \leq k \leq n-1$, then $\pi(\omega_k) = 2 \tau_s$. This is not locally admissible (the submodule is a Demazure-module, see \cite{FoL06}).
\item if $k = n$, then $\pi(\omega_n) = \omega_s$, then it is a simple case and the claim follows.
\item if $k < i$, say $i_{s-\ell} \leq k < i_{s - \ell+1} +1$, then $\pi(\omega_k) = \tau_\ell$. We have to show that 
\begin{eqnarray}
 U(\lie a \otimes \bc[t]).w \cong_{\lie a} V^{\lie a}(\tau_\ell) \oplus V^{\lie a}(\tau_{\ell-2}) \oplus \ldots \oplus V^{\lie a}(\tau_{\epsilon}).
\label{eq:lem1}
\end{eqnarray}
$V(\tau_{\ell})$ is generated  as a $\lie a$-module through $w$. Then $\tau_{\ell} - \tau_{\ell - 2} = \epsilon_{i_{s - \ell+1}} + \epsilon_{i_{s - \ell}}$. We have (since $w \in  W^{\lie g}(\psi)$), by Theorem~\ref{classical-decom}:
\[
(x^{-}_{ \epsilon_{i_{s - \ell}} + \epsilon_{i_{s - \ell - 1}}} \otimes t).w \neq 0.
\]
For weight reason this has to be a highest weight vector with respect to the $\lie a$-action. So $V^{\lie a}(\tau_{\ell -2 })$ is a direct summand of $ U(\lie a \otimes \bc[t]).w$. Iterating this gives the predicted decomposition \eqref{eq:lem1}.
\end{enumerate}

\item Suppose $\lie a$ is of type $D_s$ and $\epsilon_i \pm \epsilon_j$ correspond to the spin nodes. Let $k \in I$:
\begin{enumerate}

\item if $i \leq k  <j$, then $\pi(\omega_k) = \tau_{s-1} + \tau_s$, so by tensor product property and Theorem~\ref{classical-decom}, we have to show that 
\begin{eqnarray}
U(\lie a \otimes \bc[t]).w \cong_{\lie a}  V^{\lie a}(\tau_s) \otimes V^{\lie a}(\tau_{s-1}) \cong V^{\lie a}(\tau_{s} + \tau_{s-1}) \oplus V^{\lie a}(\tau_{s-3}) \oplus V^{\lie a}(\tau_{s-5}) \oplus \ldots
\label{eq:lem2}
\end{eqnarray}
$\tau_{s} + \tau_{s-1} - \tau_{s-3} = \beta_{s-2} + \beta_{s-1} + \beta_s = \epsilon_{i_1} + \epsilon_i$. Then we have (since $w \in W^{\lie g}(\psi)$), by Theorem~\ref{classical-decom}:
\[
(x_{\epsilon_{i_1} + \epsilon_i}^- \otimes t).w \neq 0
\]
Again, by weight reasons, this has to be a $\lie a$ highest weight vector, and, again iterating this, gives the decomposition \eqref{eq:lem2}.

\item if $j \leq k <n $, then $\pi(\omega_k) = 2 \omega_s$, so by the tensor product property and Theorem~\ref{classical-decom}, we have to show that 
\begin{eqnarray}
U(\lie a \otimes \bc[t]).w \cong_{\lie a}  V^{\lie a}(\tau_s) \otimes V^{\lie a}(\tau_{s}) \cong V^{\lie a}(2\tau_{s}) \oplus V^{\lie a}(\tau_{s-2}) \oplus V^{\lie a}(\tau_{s-4}) \oplus \ldots
\label{eq:lem3}
\end{eqnarray}
$2 \tau_s - \tau_{s-2} = \beta_s = \epsilon_i + \epsilon_j$.  Then we have (since $w \in W^{\lie g}(\psi)$), by Theorem~\ref{classical-decom}:
\[
(x_{\epsilon_{i} + \epsilon_j}^-\otimes t).w \neq 0
\]
Again, by weight reasons, this has to be a $\lie a$ highest weight vector, and, again iterating this, gives the decomposition \eqref{eq:lem3}.

\item if $k = n$, then $\pi(\omega_n) = \omega_s$ it is a simple case.

\item if $k  < i$, then $\pi(\omega_k) = \tau_\ell$ for some $\ell < s-2$, so by Theorem~\ref{classical-decom}, we have to show that 
\begin{eqnarray}
U(\lie a \otimes \bc[t]).w \cong_{\lie a}  V^{\lie a}(\tau_{\ell}) \oplus V^{\lie a}(\tau_{\ell-2}) \oplus V^{\lie a}(\tau_{\ell-4}) \oplus \ldots
\label{eq:lem4}
\end{eqnarray}
The same argument as above gives with $\tau_k - \tau_{k-2} = \epsilon_{i_{s-k}} + \epsilon_{i_{s-k+1}}$ the predicted decomposition \eqref{eq:lem4}.
\end{enumerate}

\item Suppose $\lie a$ is of type $A_s$ and $\epsilon_i + \epsilon_j \in \Pi_{\lie a}$. Let $k \in I$:
\begin{enumerate}
\item if $ k = n$ or $\min\{i_\ell, j_{s- \ell -1}\} \leq k < \max\{i_\ell, j_{s- \ell - 1} \}$, then it is a simple case.

\item if $\max\{i_\ell, j_{\ell'}\} \leq k < n$, then $\pi(\omega_k) = \tau_p + \tau_q$ for some $p \leq q$. So by the tensor product property and Theorem~\ref{classical-decom}, we have to show that 
\begin{eqnarray}
U(\lie a \otimes \bc[t]).w \cong_{\lie a}  V^{\lie a}(\tau_p) \otimes V^{\lie a}(\tau_{q}) \cong V^{\lie a}(\tau_p + \tau_q) \oplus V^{\lie a}(\tau_{p-1} + \tau_{q+1}) \oplus \ldots
\label{eq:lem5}
\end{eqnarray}
$\tau_p + \tau_q - (\tau_{p-1} + \tau_{q+1}) = \beta_p +\ldots + \beta_q = \epsilon_{i_p} + \epsilon_{j_q}$ for certain $i_p, j_q$. Then we have, since $w \in W^{\lie g}(\psi)$,
\[
(x_{\epsilon_{i_p} + \epsilon_{i_q}}^- \otimes t).w \neq 0.
\]
Again, by weight reasons, this has to be a $\lie a$ highest weight vector, and, again iterating this, gives the decomposition \eqref{eq:lem5}.
\end{enumerate}

\item Finally if $\lie a$ is of type $A_s$ and $\epsilon_i + \epsilon_j \notin \Pi_{\lie a}$ for all $i \leq j$ it is a simple case.
\end{enumerate}

\noindent
For $\lie g$ of type $D_n$ the proof is similar to the $B_n$-case.

\noindent
Suppose $\lie g$ is of type $C_n$.  We will check the various cases (see Appendix~\ref{a2} for notations):
\begin{enumerate}
\item If $\lie a$ is of type $C_s$, then $\pi(\omega_k) = \tau_p$ for some $p \in J$ it is a simple case.

\item If $\lie a$ is of type $A_s$ and $\epsilon_i + \epsilon_j \in \Pi_{\lie a}$, then
\begin{enumerate}
\item if $  \max\{i_\ell, j_{s-\ell -1}\} \leq k \leq n $ then $\pi(\omega_k) = \tau_p + \tau_q$ for some $p \leq q$. In this case $(\lie a, \omega_k)$ is not locally admissible.
\item  if $ \min \{i_\ell, j_{s-\ell-1}\} \leq k < \max \{i_\ell, j_{s-\ell-1}\}$ it is a simple case.
\end{enumerate}

\item if $\lie a$ is of type $A_s$ and $\epsilon_i + \epsilon_j \notin \Pi_{\lie a}$ for all $ i \leq j$ it is a simple case. 
\end{enumerate}
\end{proof}

\noindent
One can deduce from this list that if $(\lie a, \omega_k)$ is not locally admissible, then the restriction of a fundamental local Weyl module is not a Weyl module for $\lie a \otimes \bc[t]$.\\

\noindent
The proof of the theorem is an exhausting case-by-case consideration. The proof in the more general case of $\lie a$ being a semi-simple Levi subalgebra (for a list of all semi-simple Levi subalgebras see Table 9 in \cite{Dyn52}) uses similar arguments and we would like to omit the computations. Nevertheless, we will use this more general result of Theorem~\ref{main-thm1} in the following section in the case where $\lie a$ is generated by simple root vectors (of $\lie g$). For sake of completeness we include here the proof of this case. Note that the local Weyl modules for $(\lie a_1 \oplus \lie a_2) \otimes \bc[t]$ are isomorphic to the tensor product of local Weyl modules for $\lie a_1 \otimes \bc[t]$ and local Weyl modules for $\lie a_2 \otimes \bc[t]$).
\begin{proof} Suppose $\lie a$ is a semi-simple Levi subalgebra of $\lie g$ and generated by simple root vectors (of $\lie g$). Let $\lie a = \lie a_1 \oplus \ldots \oplus \lie a_m$ be a decomposition into simple Lie algebras. Then $\lie a_i$ is of type $X_i$ where $X_i \in \{A_\ell, B_\ell, C_\ell, D_\ell \}$. The assumption that $\lie a$ is generated by simple root vectors implies that the Dynkin diagram of $\lie a$ is a full subdiagram of the Dynkin diagram of $\lie g$, e.g. the diagram obtained through the vertices corresponding to the simple root vectors generating $\lie a$ and all arrows between two of these vertices.\\ 
An inspecting of Table Fin in Chapter 4 of \cite{Kac90} shows that all but at most one of the connected components of the subdiagram are simply-laced strings. This implies that the corresponding simple Lie algebra is of type $A$. Hence there is at most one $\lie a_i$ which is not of type $A$, in case may this be $\lie a_m$. \\ 
We denote in the following by $\pi_i$ the projection $\lie h^* \longrightarrow \lie h_i^*$. Let $(\lie a, \omega_k)$ be locally admissible, $\psi \in \mathcal{E}^{\omega_k}$ and $w$ a highest weight generator. Then by the previous proof we know that $U(\lie a_i \otimes \bc[t]).w$ is isomorphic to the local Weyl module of highest loop  weight $\pi_i(\psi)$ (resp. highest $\lie a$ weight $\pi(\omega_k)$) for $\lie a_i \otimes \bc[t]$.\\
Furthermore, by the universal property of local Weyl modules (Remark~\ref{remark-local}) we know that $U(\lie a \otimes \bc[t]).w$ is a quotient of $W^{\lie a}(\pi(\psi)) \cong  W^{\lie a_1}(\pi_1(\psi)) \otimes \ldots \otimes W^{\lie a_m}(\pi_m(\psi))$. To prove that this is not a proper quotient but actually an isomorphism it suffices to compare the dimensions.\\
We will prove the claim by induction, and assume that for $j \leq m$:
\[
U((\lie a_j  \oplus \ldots \oplus \lie a_m) \otimes \bc[t]).w \cong_{(\lie a_j  \oplus \ldots \oplus \lie a_m) \otimes \bc[t]} W^{\lie a_j}(\pi_j(\psi)) \otimes \ldots \otimes W^{\lie a_m}(\pi_m(\psi)) \subset W^{\lie g}(\psi).
\]
Let us denote $W^{\lie a_j}(\pi_j(\psi)) \otimes \ldots \otimes W^{\lie a_m}(\pi_m(\psi))$ by $W$ for short.
Since $\lie a_{j-1}$ is of type $A$ for $j \leq m$ and hence 
\[
W^{\lie a_{j-1}}(\pi_{j-1}(\omega_k))  \cong_{\lie a_{j-1} \otimes \bc[t]} V^{\lie a_{j-1}}(\pi_{j-1}(\psi)) \cong_{\lie a_{j-1}} V^{\lie a_{j-1}}(\pi_{j-1}(\omega_k)),
\]  
(where $V^{\lie a_{j-1}}(\pi_{j-1}(\omega_k))$ denotes the simple $\lie a_{j-1}$-module of highest weight $\pi_{j-1}(\omega_k))$ we have to show that
\begin{eqnarray}
U(\lie a_{j-1} \otimes \bc[t])U((\lie a_j  \oplus \ldots \oplus \lie a_m) \otimes \bc[t]).w \cong_{\lie a_{j-1}}   V^{\lie a_{j-1}}(\pi_{j-1}(\omega_k)) \otimes W.
\label{eq:semi-weyl}
\end{eqnarray}
Since $[\lie a_{j-1}, \lie a_{j+\ell}] = 0$ for $\ell \geq 0$, the right hand side is isomorphic (as a $\lie a_{j-1}$-module) to $\dim W$ copies of $V^{\lie a_{j-1}}(\pi_{j-1}(\omega_k))$. On the other hand,  let $u \in U((\lie a_j  \oplus \ldots \oplus \lie a_m) \otimes \bc[t])$ such that $u.w \neq 0$, then we have a sequence of $\lie a_{j-1} \otimes \bc[t]$-modules:
\[
W^{\lie a_{j-1}} (\pi_{j-1}(\psi)) \twoheadrightarrow U(\lie a_{j-1} \otimes \bc[t])u.w \twoheadrightarrow V^{\lie a_{j-1}}(\pi_{j-1}(\psi)).
\]
Now, as $W^{\lie a_{j-1}} (\pi_{j-1}(\psi)) \cong_{\lie a_{j-1} \otimes \bc[t]} V^{\lie a_{j-1}}(\pi_{j-1}(\psi)) \cong_{\lie a_{j-1}} V^{\lie a_{j-1}}(\pi_{j-1}(\omega_k)) $, we see that the left hand side of \eqref{eq:semi-weyl} is isomorphic (again as a $\lie a_{j-1}$-module) to $\dim W$ copies of $V^{\lie a_{j-1}}(\pi_{j-1}(\omega_k))$. This gives the induction step and finishes the proof.
\end{proof}

\subsection{}
The rest of the section is dedicated to the proof of Theorem~\ref{main-thm2}(i). So from now on we will assume that $\Pi_{\lie a} \subset \Pi_{\lie g}$, so $\Pi_{\lie a} = \{ \alpha_{i_1}, \ldots, \alpha_{i_s}\}$. Then we have the following crucial proposition:
\begin{prop}\label{prop-cruc}
Let $W$ be a finite-dimensional $\lie g \otimes \bc[t]$-module such that there exists a weight vector $w \in W$ with $U(\lie{n}^- \otimes \bc[t]).w = W$, and let $\lambda$ be the weight of $w$. Then
\[
U(\lie a \otimes \bc[t]).w =  \sum_{ \mu \in  Q^+_{\lie a}}W_{\lambda - \mu}
\]
\end{prop}
\begin{proof}
We have $R^+ = R^+_{\lie a} \cup R'$ and $\lie n^- = \lie n^-_{\lie a} \oplus \lie n^-_c$, where $\lie n^-_c$ is spanned by all root vectors $x_\alpha^-$, with $\alpha \in R'$. With the PBW Theorem we have
\[
W = U(\lie n^- \otimes \bc[t]).w = U(\lie n^-_{\lie a} \otimes \bc[t]).w_\lambda + U(\lie n^-_{\lie a} \otimes \bc[t]) U(\lie n^-_c \otimes \bc[t])_{> 0}.w_\lambda
\]
The weights in the second summand are all of the form $\lambda - \gamma_{\lie a} - n_i \alpha_i$, where $\gamma_{\lie a} \in Q^{+}_{\lie a}$ and $n_i >0$. This implies for $\mu \in Q^{+}_{\lie a}$
\[
W_{\lambda - \mu} = (U(\lie n^-_{\lie a} \otimes \bc[t]).w)_{\lambda - \mu}
\]
and this gives the proof.
\end{proof}

\subsection{}
For a given $\mu \in P^+$ and $c \in \bc$, we denote by $\mu_c \in \mathcal{E}^{\mu}$ the function defined by $\mu_c(b) = \delta_{cb}\mu$. For a given $\psi \in \mathcal{E}^{\lambda}, \lambda \in P^+$,  we can write
\[
\psi = \sum_{a \in \supp(\psi)} \psi(a)_a.
\]
The following proposition shows that for $\psi \in \mathcal{E}$, the dimension of $U(\lie a \otimes \bc[t])w_{\psi}$ depends on the classical weight of $\psi$ only but not on the support. Namely for $\psi_1, \psi_2 \in \mathcal{E}^\lambda$ we can conclude
\[ \dim U(\lie a \otimes \bc[t])w_{\psi_1} = \dim U(\lie a \otimes \bc[t])w_{\psi_2}. \]

\begin{prop}\label{prop-filtration-res}
Let $\lambda \in P^+$ and $\psi \in \mathcal{E}^\lambda$. Denote by $w_{\lambda}$ a generator of the local graded Weyl module $W^{\lie g}(\lambda_0)$ and for $a \in \supp(\psi)$, denote by $w_a$ a generator of $W^{\lie g}(\psi(a)_{a})$. Then
\[
\dim U(\lie a\otimes \bc[t]).w_\lambda  = \dim \bigotimes_{a \in \supp(\psi)} U(\lie a \otimes \bc[t]).w_a.
\]
\end{prop}

\begin{proof}
Recall from Section~\ref{fusion-product} the definition of the fusion product as the associated graded module of a tensor product. We know from Corollary~\ref{local-fusion} that $W^{\lie g}(\lambda_0)$ is isomorphic to the fusion product of 
\[ 
\bigotimes_{a \in \supp(\psi)} W^{\lie g}(\psi(a)_a).
\]
Using Proposition~\ref{prop-cruc} for $W^{\lie g}(\lambda_0)$ and the generator $w_\lambda$ we have:
\[
U(\lie n^{-}_{\lie{a}} \otimes \bc[t]).w_\lambda = \sum_{\tau \in - Q^+_{\lie a}} U(\lie n^- \otimes \bc[t])_{\tau}.w_\lambda
\]
The $\lie g$-module structures on $\bigotimes\limits_{a \in \supp(\psi)} W^{\lie g}(\psi(a)_a)$ and $W^{\lie g}(\lambda_0)$ are isomorphic (thanks to Lemma~\ref{pointless}), which implies that for all $\tau \in Q^+_{\lie a}$
\[
\dim W^{\lie g}(\lambda_0)_{\lambda- \tau} = \dim \;  \left(\bigotimes_{a \in \supp(\psi)} W^{\lie g}(\psi(a)_a)\right)_{\lambda - \tau}.
\]
The right hand side equals
\[
\dim \; \left(\bigotimes_{a \in \supp(\psi)} U(\lie n^- \otimes \bc[t]). w_a\right)_{\lambda- \tau}.
\]
With $\underline{\tau} = (\tau_a)_{a \in \supp(\psi)}$, $\tau_a \in -Q^+_{\lie a}$ and $\wt \underline{\tau} = \sum\limits_{a \in \supp(\psi)} \tau_a$, we have that this weight space equals
\[
\sum_{\wt (\underline{\tau}) = \tau} \; \;  \bigotimes_{a \in \supp(\psi)} U(\lie n^- \otimes \bc[t])_{\tau_a}. w_a = \sum_{\wt (\underline{\tau}) = \tau} \; \; \bigotimes_{a \in \supp(\psi)} U(\lie n_{\lie a}^- \otimes \bc[t])_{\tau_a}. w_a
\]
where the equation follows again from Proposition~\ref{prop-cruc}. The proposition follows now since
\[\sum_{\tau \in -Q^+_{\lie a}} U(\lie n^-_{\lie a} \otimes\bc[t])_{\tau} = U(\lie n_{\lie a}^- \otimes \bc[t]).\]
\end{proof}

We are ready to prove Theorem~\ref{main-thm2}(i): the highest weight component of a local Weyl module $W^{\lie g}(\psi)$ is isomorphic to a local Weyl module for $\lie a \otimes \bc[t]$ if $(\lie a, \wt \psi)$ is locally admissible. 
\begin{proof}
The universal property of the local Weyl module (Remark~\ref{remark-local}) gives for any $\psi \in \mathcal{E}^\lambda$ a surjective map of $\lie a \otimes \bc[t]$-modules
\[
W^\lie a(\pi(\psi)) \longrightarrow U(\lie a \otimes \bc[t])w_\psi \subset W^\lie g(\psi).
\]
To prove the theorem, it is sufficient to show that the dimensions on both sides are equal. \\
Recall here our previous observations from Section~\ref{section-local}, namely that for $\psi = \psi_1 + \ldots + \psi_k$, $\supp (\psi_i) = \{c_i \}$ and $c_i \neq c_j$ (for $i \neq j$) we have
\[
 W^\lie g(\psi) \cong U(\lie g \otimes \bc[t]/(t-c_1)^{N_1} \bc[t]). w_{\psi_1} \bigotimes \ldots \bigotimes U(\lie g \otimes \bc[t]/(t-c_s)^{N_s} \bc[t]).w_{\psi_s}\]
for $N_i$ big enough.\\
We see that the surjective map 
\[
W^\lie a(\pi(\psi)) \longrightarrow U(\lie a \otimes \bc[t])w_\psi \subset W^\lie g(\psi)
\]
is actually a surjective map (which is also component wise surjective)
\[
\bigotimes_{j=1 }^k W^\lie a(\pi(\psi_j)) \longrightarrow \bigotimes_{j = 1}^k  U(\lie a \otimes \bc[t])w_{\psi_j}.
\]
To compare dimensions it is therefore sufficient to compare dimensions for local Weyl modules supported in a single point only. Again with Lemma~\ref{pointless}, we can choose this point arbitrarily, so for simplicity we choose the origin.\\
 So let $\lambda = \omega_{i_1} + \ldots+ \omega_{i_s} \in P^+$ and $w_\lambda \in W^{\lie g}(\lambda_0)_\lambda$. We fix $\psi \in \mathcal{E}^{\lambda}$ such that $\psi(a) = \omega_{i_j}$ for some $j$ (for all $a \in \supp(\psi)$ and $\psi(b) = 0$ else). Then Proposition~\ref{prop-filtration-res} implies that 
\[
\dim U(\lie a\otimes \bc[t]).w_\lambda  = \prod_{a \in \supp(\psi)} \dim U(\lie a \otimes \bc[t]).w_a.
\]

\noindent
To know the left hand side it is enough to know each factor of the right hand side. So we pick an arbitrary factor of the right hand side and may again assume without loss of generality (Lemma~\ref{pointless}) that $a = 0$. So we have to understand
\[
\dim U(\lie a \otimes \bc[t]).w
\]
where $w$ is a generator of $W^{\lie g}((\omega_{i_j})_0)_{\omega_{i_j}}$. Lemma~\ref{funda-res} implies that if $(\lie a, \omega_{i_j})$ is locally admissible then
\begin{eqnarray}
\dim U(\lie a \otimes \bc[t]).w = \dim W^{\lie a}(\pi(\omega_{i_j})_0).
\label{eq:last-eq}
\end{eqnarray}

\noindent
Since $(\lie a, \lambda)$ is locally admissible, $(\lie a, \omega_{i_j})$ is locally admissible for all $j$ and hence
\[
\dim U(\lie a\otimes \bc[t]).w_\lambda  = \prod_{j=1}^{s}\dim W^{\lie a}(\pi(\omega_{i_j})_0)
\]
But this is nothing but the dimension of  $W^{\lie a}(\pi(\psi))$ thanks to Theorem~\ref{global-free} and the tensor product property (Theorem~\ref{tensor}), applied in the context of the current algebra $\lie a \otimes \bc[t]$. This finishes the proof.
\end{proof}

%%%%%%%%%%%%%%%%%%%%%%%%%%%%%%%%%%%%%%%%%%
%%%%%%%%%%%%%%%%%%%%% Global Weyl modules %%%%%%%%%%

\section{Global Weyl modules and subalgebras}\label{section-5}
In order to prove Theorem~\ref{main-thm2}(ii) we will need the following two results:
\begin{prop}\label{global-quot}
Let $\lambda \in P^+$, then the assignment $w_{\pi(\lambda)} \mapsto w_\lambda$ induces a map of $\lie a \otimes \bc[t]$-modules
\[ 
W^{\lie a}(\pi(\lambda)) \longrightarrow W^{\lie g}(\lambda).
\]
\end{prop}
\noindent
Theorem~\ref{main-thm2}(ii) claims that this induced map is injective.
\begin{proof}
We have to check that the defining relations of $W^{\lie a}(\pi(\lambda))$ (see Definition~\ref{global-defn}) are satisfied by $w_\lambda$. But this is obvious since the $\lie a$-weight of $w_\lambda$ is $\pi(\lambda)$ and the triangular decompositions of $\lie a$ and $\lie g$ are compatible.
\end{proof}

\noindent
We have an induced map of algebras
\begin{eqnarray}
\ba_{\pi(\lambda)}^{\lie a} \longrightarrow \ba_{\lambda}^{\lie g}
\label{existence-A}
\end{eqnarray}
which should be studied in more detail:

\begin{prop}\label{injective-A} If $(\lie a, \lambda)$ is globally admissible, then the induced map
\[
\ba_{\pi(\lambda)}^{\lie a} \longrightarrow \ba_{\lambda}^{\lie g}.
\]
is injective.
\end{prop}

\noindent
As for Proposition~\ref{surjective-E}, it is clear that the map is not injective if $(\lie a, \lambda)$ is not globally admissible.
\begin{proof}
By Theorem~\ref{global-free} we have to show 
\[
\bigcap_{ \phi \in \mathcal{E}^{\pi(\lambda)}} \ann_{U(\lie h_{\lie a} \otimes \bc[t])} (v_\phi) =  U(\lie h_ {\lie a} \otimes \bc[t])  \cap \bigcap_{ \psi \in \mathcal{E}^{\lambda} } \ann_{U(\lie h \otimes \bc[t])} (v_\psi). 
\]
The left hand side is clearly contained in the right hand side by \eqref{existence-A}. Let $u \in U(\lie h_{\lie a} \otimes \bc[t])$ such that 
\[
u  \notin \bigcap_{ \phi \in \mathcal{E}^{\pi(\lambda)}} \ann_{U\lie h_{\lie a} \otimes \bc[t])} (v_\phi).
\] 
Then there exists $\phi \in \mathcal{E}^{\pi(\lambda)}$ such that $u.v_{\phi} \neq 0$ in $V^\lie a(\phi)$. Since $(\lie a, \lambda)$ is globally admissible, we know by Proposition~\ref{surjective-E} that there exists $\psi \in \mathcal{E}^{\lambda}$ such that
\[
\pi(\psi) = \phi \text{ and } V^{\lie a}(\phi) \cong_{\lie a \otimes \bc[t]} U(\lie a \otimes \bc[t]).v_\psi
\]
This implies that $u.v_\psi \neq 0$ in $V^\lie g(\psi)$ and hence 
\[
u \notin \bigcap_{ \psi \in \mathcal{E}^{\lambda}} \ann_{U(\lie h \otimes \bc[t])} (v_\psi)
\]
which proves the proposition.
\end{proof}

From now on, we assume that $\Pi_{\lie a} \subset \Pi$. This implies that $\lie a$ is generated by root vectors of simple $\lie g$-roots.
\begin{thm} 
Let $\lambda \in P^+$ be such that $(\lie a, \lambda)$ is globally admissible. Let $w_\lambda$ be a highest weight generator of $W^{\lie g}(\lambda)$, then
\[
U(\lie a \otimes \bc[t]).w_\lambda \cong_{\lie a \otimes \bc[t]} W^{\lie a}(\pi(\lambda)).
\]
\end{thm}
\begin{proof}
The universal property of the global Weyl modules (see Remark~\ref{remark-global}) implies that left hand side is a quotient of the right hand side.\\ 
We will show first that $U(\lie a \otimes \bc[t]).w_\lambda$ is a right $\ba_{\pi(\lambda)}^{\lie a}$-module. Since $W^{\lie g}(\lambda)$ is a right $\ba_\lambda^{\lie g}$-module, with \eqref{existence-A} we obtain an action of $\ba_{\pi(\lambda)}^{\lie a}$ on $W^{\lie g}(\lambda)$. But clearly $U(\lie a \otimes \bc[t]).w_\lambda$ is stable under this action, since it is stable under $U(\lie h_{\lie a} \otimes \bc[t])$.

\noindent
Now, we are to show that $U(\lie a \otimes \bc[t]).w_\lambda$ is a \textit{free} $\ba_{\pi(\lambda)}^{\lie a}$-module of the same rank as the free $\ba_{\pi(\lambda)}^{\lie a}$-module-module $W^{\lie a}(\pi(\lambda))$ (see Theorem~\ref{global-free}). For this we use the same idea as in \cite{CP01}, written in full detail in \cite{FMS12}. That is to say the following: Let $\bm \in \max \ba_{\pi(\lambda)}^{\lie a}$ be a maximal ideal, one has to show that the dimension of the cyclic $\lie a \otimes \bc[t]$-module
\begin{eqnarray}
U(\lie a \otimes \bc[t]).w_\lambda \otimes_{\ba_{\pi(\lambda)}^{\lie a}} \ba_{\pi(\lambda)}^{\lie a}/\bm
\label{eq:res-global}
\end{eqnarray}
is finite and independent of the chosen maximal ideal, and equals to $ \dim W^{\lie a}(\phi) $ for all $\phi \in \mathcal{E}^{\pi(\lambda)}$, which is the rank of $W^{\lie a}(\pi(\lambda))$.\\

\noindent
Let $\bm \in \max \ba_{\pi(\lambda)}^{\lie a}$, then $\bm$ corresponds to a finitely supported function $\phi \in \mathcal{E}_{\lie a}^{\pi(\lambda)}$ (see Section~\ref{section-3}). We have
\begin{eqnarray}
W^{\lie a}(\phi) \twoheadrightarrow U(\lie a \otimes \bc[t]).w_\lambda \otimes_{\ba_{\pi(\lambda)}^{\lie a}} \ba_{\pi(\lambda)}^{\lie a}/\bm,
\label{eq:res-local-weyl}
\end{eqnarray}
this follows from the universal property (Remark~\ref{remark-local} and the fact that the right hand side is cyclic of highest current weight corresponding to $\phi$). Hence the dimension of the right hand side is finite.\\
On the other hand, since $(\lie a, \lambda)$ is globally admissible, we know by Proposition~\ref{surjective-E} that there exists $\psi \in \mathcal{E}^{\lambda}$, such that $\pi(\psi) = \phi$. This implies
\[
 U(\lie a \otimes \bc[t]).w_\lambda \otimes_{\ba_{\pi(\lambda)}^{\lie a}} \ba_{\pi(\lambda)}^{\lie a}/\bm \cong_{\lie a \otimes \bc[t]} U(\lie a \otimes \bc[t]).w_{\psi}
\]
where $w_\psi$ is a generator of the $\lie g \otimes \bc[t]$-module $W^{\lie g}(\psi)$. With Theorem~\ref{main-thm2}(i) we know
\[
U(\lie a \otimes \bc[t]).w_{\psi} \cong_{\lie a \otimes \bc[t]} W^{\lie a}(\pi(\psi)) = W^{\lie a}(\phi).
\]
So the dimension of \eqref{eq:res-global} is equals to the dimension of $W^{\lie a}(\phi)$, hence the map in \eqref{eq:res-local-weyl} is an isomorphism. This also implies that the dimension of \eqref{eq:res-global} is independent of the chosen maximal ideal. Using the argumentation in the appendix of \cite{FMS12} we can summarize:\\ 
$U(\lie a \otimes \bc[t]).w_\lambda$ is a free $\ba_{\pi(\lambda)}^{\lie a}$ module of the same rank as the global Weyl module $W^{\lie a}(\pi(\lambda))$. Hence the canonical map 
\[
W^{\lie a}(\pi(\lambda))  \twoheadrightarrow U(\lie a \otimes \bc[t]).w_\lambda
\]
is an isomorphism and the theorem is proven.
\end{proof}

%%%%%%%%%%%%%%%%%%%%%%%%%%%%%%%%%%%%%%%%%%%%%%%%%%%%%%%%
%%%%%%%%%% Appendix %%%%%%%%%%%%%%%%%%%%%%%%%%%%%%%%%%

\begin{appendix}
\section{Levi subalgebras}
We give a complete list of all simple Levi subalgebras obtained from root spaces of closed subsets of $R$.

\subsection{Roots of Lie algebras of classical type}\label{simple-sub}
 First we give a list of the positive roots of $\lie g$ (see \cite{Car05}):
\begin{itemize}
\item If $\lie g$ is of type $A_n$, then
\[
R^+= \{ \varepsilon_i - \varepsilon_j \; | \; 1 \leq i < j \leq n+1 \}.
\]
\item If $\lie g$ is of type $B_n$, then the long positive (resp. short positive) roots are
\[
(R^+)^{\ell} = \{ \varepsilon_i \pm \varepsilon_j  \; | \; 1 \leq i < j \leq n \} \text{ and } (R^+)^{s} = \{ \varepsilon_i \; | \; 1 \leq i \leq n \}.
\]
\item If $\lie g$ is of type $C_n$, then
\[
(R^+)^{\ell} = \{ 2 \varepsilon_i \; | \; 1 \leq i \leq n \} \text{ and }  (R^+)^{s} = \{ \varepsilon_i \pm \varepsilon_j  \; | \; i \leq 1 \leq i < j \leq n \}.
\]
\item If $\lie g$ is of type $D_n$, then 
\[
R^+ = \{ \varepsilon_i \pm \varepsilon_j  \; | \; 1 \leq i < j \leq n \}.
\]
\end{itemize}

\subsection{Roots of simple Levi subalgebras}\label{a2}
The following gives a description of all possible sets of simple roots $\Pi_{\lie a}$ of simple Levi subalgebras. The proofs are omitted, since the results are obtained by simple case by case considerations.
\begin{itemize}
\item In all types: If $\Pi_{\lie a} = \{ \varepsilon_{i_1} - \varepsilon_{i_2}, \varepsilon_{i_2} - \varepsilon_{i_3} , \ldots, \varepsilon_{i_s} - \varepsilon_{i_{s+1}} \}$
with $1 \leq i_1 < i_2 < \ldots < i_s < i_{s+1} \leq n+1$  if $\lie g $ is of type $A_n$ (resp. $\leq n$ if $\lie g$ is of other type) $\Rightarrow$ $\lie a$ is of type $A_s$. 
\item
Let $\lie g$ be of type $B_n$ and $\lie a \subseteq \lie g$ a simple Lie subalgebra of rank $s$, then there are several cases:
\begin{enumerate}
\item If $\varepsilon_i \in \Pi_{\lie a}$ for some $i \in I$, then $\lie a$ is of type $B_s$ and $\varepsilon_i$ is the unique simple short root. Then 
\[
\Pi_{\lie a} = \{  \varepsilon_{i_{s-1}} - \varepsilon_{i_{s - 2}}, \ldots, \varepsilon_{i_1} - \varepsilon_i, \varepsilon_{i} \}
\] 
\item If $\{ \varepsilon_{i} + \varepsilon_j, \varepsilon_{i} - \varepsilon_{j}\} \subset \Pi_{\lie a}$ for some $i < j $, then $\lie a$ is of type $D_s$ and $ \varepsilon_{i} + \varepsilon_j, \varepsilon_{i} - \varepsilon_{j}$ correspond to the spin nodes. Then
\[
\Pi_{\lie a} =  \{\varepsilon_{i_{s-2}} - \varepsilon_{i_{s - 3}}, \ldots, \varepsilon_{i_1} - \varepsilon_i, \varepsilon_{i} + \varepsilon_j, \varepsilon_{i} - \varepsilon_{j} \}
\]
\item In the remaining cases, $\lie a$ is of type $A_s$:
 If $\varepsilon_{i} + \varepsilon_{j} \in \Pi_{\lie a}$ for some $i < j$, then 
\begin{eqnarray}
 \Pi_{\lie a} = \{ \varepsilon_{i_{\ell}} - \varepsilon_{i_{\ell - 1}}, \ldots, \varepsilon_{i_1} - \varepsilon_i, \varepsilon_{i} + \varepsilon_{j} , \varepsilon_{j_1} - \varepsilon_j, \ldots, \varepsilon_{j_{s - \ell -1}} - \varepsilon_{j_{s- \ell -2}} \}
\label{type:BN-AN-BAD}
\end{eqnarray}
for some $ 1 \leq i_{\ell} < \ldots < i_1 < i, 1 \leq j_{s - \ell - 1} < \ldots < j_1 < j $ and $i_k \neq j_{k'}$ for all $k,k'$.
\end{enumerate}
 
\item
Let $\lie g$ be of type $C_n$ and $\lie a \subseteq \lie g$ a simple Lie subalgebra of rank $s$, then there are several cases:
\begin{enumerate}
\item If $2 \varepsilon_i \in \Pi_{\lie a}$ for some $i \in I$, then $\lie a$ is of type $C_s$ and $2 \varepsilon_i$ is the unique simple long root.
\item In the remaining cases $\Pi_{\lie a}$ is equals to \eqref{type:BN-AN-BAD} and $\lie a$ is of type $A_s$.
\end{enumerate}

\item
Let $\lie g$ be of type $D_n$ and $\lie a \subseteq \lie g$ a simple Lie subalgebra of rank $s$, then there are several cases
\begin{enumerate}
\item If $\{ \varepsilon_{i} + \varepsilon_j, \varepsilon_{i} - \varepsilon_{j}\} \subset \Pi_{\lie a}$ for some $i < j \leq n $, then $\lie a$ is of type $D_s$ and $ \varepsilon_{i} + \varepsilon_j, \varepsilon_{i} - \varepsilon_{j}$ correspond to the spin nodes. Then
\[
\Pi_{\lie a} =  \{  \{  \varepsilon_{i_{s-2}} - \varepsilon_{i_{s - 3}}, \ldots, \varepsilon_{i_1} - \varepsilon_i, \varepsilon_{i} + \varepsilon_j, \varepsilon_{i} - \varepsilon_{j} \}
\]
\item In the remaining cases $\Pi_{\lie a}$ is equals to \eqref{type:BN-AN-BAD} and $\lie a$ is of type $A_s$.
\end{enumerate}
\end{itemize}
\bigskip
\bigskip
\noindent
\textbf{Acknowledgments:} I would like to thank Craig~S. for helpful discussions and the anonymous referee for valuable suggestions.
\end{appendix}
\bibliographystyle{alpha}
\bibliography{weyl-sub-biblist}
\end{document}